\documentclass[12pt]{article}

\usepackage[T1]{fontenc}
\usepackage{avant}
\usepackage[cp1250]{inputenc}
\usepackage{amsmath,amssymb,amsthm}

\usepackage{dynkin-diagrams}

\usepackage[ruled,linesnumbered,norelsize]{algorithm2e}

\title{Non-existence of standard compact Clifford-Klein forms of homogeneous spaces of exceptional Lie groups}
\author{Maciej Boche\'nski, Piotr Jastrz\c ebski, Aleksy Tralle}

\begin{document}

\newtheorem{theorem}{Theorem}
\newtheorem{proposition}{Proposition}
\newtheorem{lemma}{Lemma}
\newtheorem{definition}{Definition}
\newtheorem{example}{Example}
\newtheorem{note}{Note}
\newtheorem{corollary}{Corollary}
\newtheorem{remark}{Remark}
\newtheorem{conjecture}{Conjecture}
\newtheorem{problem}{Problem}

\maketitle{}
\abstract{We use a computer-aided approach to prove that there are no standard compact Clifford-Klein forms of homogeneous spaces of exceptional Lie groups. This yields further support for Kobayashi's conjecture about possible compact Clifford-Klein forms. Our approach is based on the algorithms developed in this work. They are inspired by  the algorithmic  methods of classifying semisimple subalgebras in simple  Lie algebras developed by Faccin and de Graaf. Also, we use the databases created by them. These algorithms eliminate the majority of possibilities. We complete the proof using Kobayashi's criterion for properness of the Lie group action.}

\section{Introduction}\label{sec:introd}
Let $G$ be a semisimple linear real Lie group, $H\subset G$ a reductive (non-compact) subgroup such that $G/H$ is non-compact. We say that $G/H$ admits a compact Clifford-Klein form, if there exists a discrete subgroup $\Gamma \subset G$ acting properly, freely and co-compactly on $G/H$. The problem of determining which reductive homogeneous spaces admit such forms goes back to Calabi and Markus (\cite{cm}), Kulkarni (for some pseudo-Riemannian space forms, \cite{ku}) and was formulated as a research program by T. Kobayashi \cite{kob}. 
One of the important and challenging problems in the whole area is  Kobayashi's conjecture (see \cite{ko}, Conjecture 3.3.10). Following \cite{kas-kob} we say that $G/H$ admits a \textit{standard compact Clifford-Klein form}, if there exists a reductive Lie subgroup $L\subset G$ such that $L$ acts properly on $G/H$ and $L\backslash G/H$ is compact. Note that in this case any co-compact torsion-free lattice $\Gamma\subset L$ yields a compact Clifford-Klein form.  The Kobayashi conjecture states that for the homogeneous spaces $G/H$ of reductive type, the existence of a compact Clifford-Klein form on $G/H$ implies the existence of a standard one. Note that the conjecture does not say that all compact Clifford-Klein forms are standard. There are examples of non-standard ones (see \cite{kas}, \cite{kdef}), obtained by a deformation of a lattice $\Gamma\subset L$. 
\noindent

In fact all known examples of standard Clifford-Klein forms  can be obtained as follows. Assume that there exists a reductive Lie subgroup $L\subset G$ such that $G=L\cdot H$,  and $L\cap H$ is compact. Under these assumptions we see that 
\begin{itemize}
\item $L$ acts transitively on $G/H$ and there is a diffeomorphism $G/H\simeq L/(L\cap H),$
\item since $L\cap H$ is compact, any co-compact lattice $\Gamma\subset L$ acts properly and co-compactly on $L/(L\cap H)$ and hence on $G/H$.
\end{itemize}
Notice that  Kobayashi's conjecture indicates that compact Clifford-Klein forms of non-compact homogeneous spaces $G/H$ of reductive type are special. There are many partial results which yield a strong evidence for the conjecture. For instance there are topological obstructions \cite{kob-ono}, \cite{mor}, \cite{Th},  results of T. Kobayashi \cite{kob-d}, \cite{kob-ws}, \cite{kob-pm},  Margulis \cite{mar}, Zimmer \cite{zim}, Hee Oh and Witte \cite{ow}, Okuda \cite{ok} and the authors of this paper \cite{bjstw}, \cite{btjo}, \cite{btc}, \cite{btams}.  However, Kobayashi's question (which is challenging and mathematically interesting) is still to be answered. It should be noted that Kobayashi found a criterion for the existence of standard Clifford-Klein forms in terms of the data of $G/H$: one needs to know the non-compact dimensions of $G,H,L$ and the action of the little Weyl group of $G$ on the non-compact parts of the real split Cartan subalgebras of the Lie algebras $\mathfrak{g}$, $\mathfrak{h}$ and $\mathfrak{l}$  of $G,$ $H$, $L,$ respectively (\cite{kob}, Theorem 4.1 and Theorem 4.7). However, this approach works for the {\it given} triples $(G,H,L)$ and the {\it known} embeddings of $H$ and $L$ into $G$ (basically, described in terms of the roots systems).

\noindent
The aim of this paper is to prove the following.
\begin{theorem}
Let $G$ be a simple linear real Lie group of \textit{exceptional} type, $H\subset G$ a reductive non-compact subgroup such that $G/H$ is non-compact. Then $G/H$ does not admit standard compact Clifford-Klein forms.
\label{tw1}
\end{theorem}
To prove Theorem \ref{tw1} we translate the Kobayashi criterion into a computer algorithm to find all possible triples $(\mathfrak{g},\mathfrak{h},\mathfrak{l})$ which may induce compact Clifford-Klein forms. Our approach is inspired by  algorithmic methods  of classification of subalgebras in simple Lie algebras developed by Faccin and de Graaf \cite{dg}, \cite{dg1}, \cite{sla}.  Eventually we eliminate each triple using known criteria of existence of compact Clifford-Klein forms. 

Recently, Tojo \cite{tojo} proved that irreducible symmetric spaces $G/H$ that admit standard Clifford-Klein forms are exactly those in the list \cite{ko}. Thus, there is an overlap with our results in this case.
\vskip6pt
\noindent {\bf Acknowledgment}. We thank the anonymous referees for valuable advice and for informing us about Tojo's paper \cite{tojo}. We are grateful to them for the careful reading and pointing out a wrong claim in the first version of this article. The first named author acknowledges the support of the National Science Center (grant NCN no. 2018/31/D/ST1/00083). The third named author was supported by NCN grant  2018/31/B/ST1/00053.

\section{Preliminaries}\label{sec:prelim}
We use the basics of Lie theory without explanations, the reader may consult \cite{ov}. We consider root systems of complex semisimple Lie algebras with respect to the Cartan subalgebras. We use the notation $\mathfrak{g}$ for {\it real} Lie algebras, while $\mathfrak{g}^c$ means a complex Lie algebra (or, the complexification of $\mathfrak{g}$, the context will be always clear).  If $\mathfrak{g}=\mathfrak{k}+\mathfrak{p}$ denotes the Cartan decomposition of a semisimple non-compact real Lie algebra, then there is a maximal $\mathbb{R}$-diagonalizable subalgebra $\mathfrak{a}\subset\mathfrak{p}$, and any two such subalgebras are transformed into each other by an element in $K,$ the maximal compact subgroup of $G$ (corresponding to $\mathfrak{k}$). Moreover, $\mathfrak{a}$ is a maximal $\mathbb{R}$-diagonalizable subalgebra in $\mathfrak{g}$, and all such subalgebras are conjugate. We set $\operatorname{rank}_{\mathbb{R}}\mathfrak{g}:= \dim\, \mathfrak{a},$ the real rank of $\mathfrak{g}.$ Let $N_K(\mathfrak{a})$ and $Z_K(\mathfrak{a})$ denote the normalizer and the centralizer of $\mathfrak{a}$ in $K$. The {\it little Weyl group} of $G$ is, by definition, the group $W=N_K(\mathfrak{a})/Z_K(\mathfrak{a})$.

 We will need the notion of the \textit{a-hyperbolic rank} \cite{bo}, \cite{btc} which will be used in the sieving procedure eliminating some possibilities for compact Clifford-Klein forms. Let  $\mathfrak{t}$ be a split Cartan subalgebra in $\mathfrak{g}$ containing $\mathfrak{a}$, and $\Sigma\subset\mathfrak{a}^*$ be a {\it restricted} root system of $\mathfrak{g}$. Choose a subset of positive roots $\Sigma^{+}\subset \Sigma.$ We have
$$\mathfrak{g}=\mathfrak{k}+\mathfrak{a}+\mathfrak{n}, \ \ \mathfrak{n}:= \sum_{\alpha\in\Sigma^{+}} \mathfrak{g}_{\alpha}.$$
The above decomposition is called the \textit{Iwasawa decomposition}. On the Lie group level we have
$$G=KAN.$$
where the subgroup $A$ is simply connected abelian, $N$ is simply connected unipotent and $A$ normalizes $N.$ Also $N$ is the maximal unipotent subgroup of $G,$ any unipotent subgroup of $G$ is conjugate to a subgroup of $N.$

Define 
 $$\mathfrak{a}^+=\{X\in\mathfrak{a}\,|\,\xi(X)\geq 0,\forall \xi\in\Sigma^+\}.$$
Consider the Weyl group $W_{\mathfrak{g}}$ of $\mathfrak{g}$ and choose the longest element $w_0\in W_{\mathfrak{g}}$. Define $-w_0:\mathfrak{a}\rightarrow\mathfrak{a}$ by the formula $X\rightarrow -(w_0X)$. It can be checked that $-w_{0} $ has the property $-w_{0}(\mathfrak{a}^+)\subset\mathfrak{a}^+$. Let 
 $$\mathfrak{b}^+=\{X\in\mathfrak{a}^+\,|\,-w_{0}(X)=X\}.$$
 \begin{definition} {\rm The dimension of $\mathfrak{b}^+$ (as a convex cone) is called the {\it a-hyperbolic rank} of $\mathfrak{g}$. It is denoted by}
 $$\operatorname{rank}_{\operatorname{a-hyp}}\mathfrak{g}.$$ 
 \end{definition}
 Here is an example which shows the way of using this invariant.
 \begin{theorem}[\cite{btc}, Theorem 8]\label{thm:a-hyp} Let $G$ be a connected and semisimple linear Lie group and $H$ a reductive subgroup with finite number of connected components. Then
 \begin{itemize}
 \item if $\operatorname{rank}_{\operatorname{a-hyp}}\mathfrak{g}=\operatorname{rank}_{\operatorname{a-hyp}}\mathfrak{h}$, then $G/H$ does not admit discontinuous actions of non virtually abelian discrete subgroups (and, therefore, compact Clifford-Klein forms),
 \item If $\operatorname{rank}_{\operatorname{a-hyp}}\mathfrak{g}>\operatorname{rank}_{\mathbb{R}}\mathfrak{h}$, then $G/H$ admits a discontinuous action of a non virtually abelian discrete subgroup.
 \end{itemize}
 \end{theorem}
 Note that calculations of $\operatorname{rank}_{\operatorname{a-hyp}}\mathfrak{g}$ are done in \cite{btc}. 

 If $G$ is a connected linear real Lie group with the semisimple real Lie algebra $\mathfrak{g}$, the symbol $G^c$ denotes the connected complex Lie group corresponding to $\mathfrak{g}^c$ and such that $G\subset G^c.$ We say that an element $X \in \mathfrak{g}^c$ is {\it hyperbolic}, if $\operatorname{ad}_{X}$ is diagonalizable and all eigenvalues of $\operatorname{ad}_{X}$ are real.
\begin{definition}
{\rm An adjoint orbit $O_{X}:=\operatorname{Ad}_{{G}^c}(X) \subset \mathfrak{g}^c$  is said to be hyperbolic if $X$ is hyperbolic. We use the same terminology for the the orbit of $\operatorname{Ad}_G$ in $\mathfrak{g}$. } 
\end{definition}  
 In this paper we work with semisimple subalgebras in simple real and complex Lie algebras and our approach is based on methods of computer aided classification of embeddings of Lie algebras \cite{dg}, \cite{dg1},\cite{sla}. This classification is considered up to equivalence or linear equivalence.
 \begin{definition}[\cite{dg}, \cite{dg1}] {\rm Let $\mathfrak{g}$ and $\tilde{\mathfrak{g}}$ be two Lie algebras. Two embeddings $\varepsilon:\mathfrak{g}^c\hookrightarrow\tilde{\mathfrak{g}}^c,$ $\varepsilon ':\mathfrak{g}^c\hookrightarrow\tilde{\mathfrak{g}}^c$ are called {\it equivalent} if there is an inner automorphism $\phi\in\operatorname{Aut}(\tilde{\mathfrak{g}}^c)$ such that $\varepsilon=\phi\varepsilon'$. They are called {\it linearly equivalent} if for all representations $\rho:\tilde{\mathfrak{g}}^c\rightarrow \mathfrak{gl}(V^c)$ the induced representations $\rho\circ\varepsilon$ and $\rho\circ\varepsilon'$ are equivalent.}
 \end{definition}
 The equivalence implies linear equivalence, but the converse is false (see, for example, \cite{min}). However,  it is mentioned in \cite{dg} that the coincidence of these classes is ubiquitous if $\tilde{\mathfrak{g}}^{c}$ is an exceptional simple Lie algebra and $\mathfrak{g}^c$ is semisimple. 
 
    In Section \ref{sec:proof} we consider {\it symmetric pairs} $(\mathfrak{g},\mathfrak{h})$, that is, when $\mathfrak{h}$ is the subalgebra of all fixed points of an involutive automorphism of $\mathfrak{g}$. They correspond to pseudo-Riemannian symmetric spaces.

\section{Kobayashi's criterion for proper actions and sieving algorithms}\label{sec:sewing}
\subsection{A criterion for properness}
Let $G$ be a simple linear real Lie group and $H,L\subset G$ reductive subgroups of $G.$ Choose  a Cartan decomposition
\begin{equation}
 \mathfrak{g}=\mathfrak{k} + \mathfrak{p},
 \label{eq1}
\end{equation}
where $\mathfrak{k}$ is the Lie algebra of the maximal compact subgroup $K$ of $G.$
By assumption, $\mathfrak{h}$ and $\mathfrak{l}$ are reductive subalgebras of $\mathfrak{g}$, therefore, they admit Cartan decompositions 
$$\mathfrak{h}=\mathfrak{k}_{h} + \mathfrak{p}_{h} \ \text{and} \ \mathfrak{l}=\mathfrak{k}_{l} + \mathfrak{p}_{l},$$
such that $\mathfrak{p}_{h}, \mathfrak{p}_{l} \subset \mathfrak{p}$ (see \cite{kob} for more details). One can choose  maximal abelian subalgebras
$$\mathfrak{a} \subset \mathfrak{p}, \ \ \mathfrak{a}_{h} \subset \mathfrak{p}_{h}, \ \ \textrm{and} \ \ \mathfrak{a}_{l} \subset \mathfrak{p}_{l}$$
such that $\mathfrak{a}_{h},\mathfrak{a}_{l} \subset \mathfrak{a}.$
Let us recall the definition of the proper group action. Let $S$ be a locally compact topological group acting continuously on a locally compact Hausdorff topological space $X$. This action is {\it proper} if for every compact subset $S \subset X$ the set
$$L_{S}:=\{  g\in L \ | \ g\cdot S \cap S \neq \emptyset \}$$
is compact. Kobayashi proved the following criterion for the properness of a Lie group action.
\begin{theorem}[\cite{kob}, Theorem 4.1] The following three conditions are equivalent
 \begin{itemize}
   \item $L$ acts on $G/H$ properly.
   \item $H$ acts on $G/L$ properly.
   \item For any $w \in W$ (where $W$ denotes the little Weyl group of $\mathfrak{g}$ )
   $$(w\cdot \mathfrak{a}_{l}) \cap \mathfrak{a}_{h} =\{ 0 \}.$$ 
 \end{itemize}
 \label{twkob}
\end{theorem}

\noindent
Let 
$$d(G):=\dim\,\mathfrak{p},\,d(H):=\dim\,\mathfrak{p}_{h},\,  d(L):=\dim\,\mathfrak{p}_{l}.$$

\begin{corollary}[\cite{kob}, Theorem 4.7]
If a triple $(G,H,L)$ induces a standard compact Clifford-Klein form then 
\begin{equation}
d(G)=d(H)+d(L).
\label{eq2}
\end{equation}
Conversely, if $L$ acts properly on $G/H$ and ($\ref{eq2}$) is satisfied, then $(G,H,L)$ determines a standard compact Clifford-Klein form. Moreover for any $g\in G$ we have
$$gHg^{-1}\cap L \subset K',$$
where $K'\subset G$ is a conjugate of $K.$
\label{wnkob}
\end{corollary}
Kobayashi's  criterion can be reformulated in the language of hyperbolic orbits \cite{ok}. This  will be used in Section 4.
\begin{proposition}[\cite{ok}, Theorem 4.1]
Let $H,L$ be reductive subgroups of a semisimple real Lie group $G$. The subgroup $L$ does not act properly on $G/H$ if and only if there exists a non-trivial hyperbolic orbit $Ad_{G}(X)$ meeting $\mathfrak{h}$ and $\mathfrak{l} .$
\label{kobhyp}
\end{proposition}

\noindent
We also need the following result.
\begin{theorem}[\cite{bela}, Corollary 3, see also \cite{kob-d}]\label{thm:labourie}
If $G/H$ admits a compact Clifford-Klein form then the center of $H$ is compact.
\end{theorem}

\noindent
We will use the following easy corollary to the previous results.
\begin{proposition}
Assume that for any non-compact semisimple subgroup $H'\subset G$ the non-compact homogeneous space $G/H'$ does not admit a standard compact Clifford-Klein form. Then for any non-compact reductive subgroup $H\subset G$ the non-compact homogeneous space $G/H$ does not admit standard compact Clifford-Klein forms.
\label{pro1}
\end{proposition}
\begin{proof}
Assume that $G/H$ admits a standard compact Clifford-Klein form for some reductive subgroup $H\subset G$ so there exists a triple $(G,H,L).$ We see that $G/L$ and $G/H$ admit (standard) compact Clifford-Klein forms. Thus the centers of $H$ and $L$ are compact by Theorem \ref{thm:labourie}. Let $H'\subset H$ and $L'\subset L$ be the semisimple parts of $H$ and $L,$ respectively. We have
$$d(H)=d(H'), \ \ \ d(L)=d(L')$$
and so $(G,H',L')$ induces a standard compact Clifford-Klein form.
\end{proof}

\subsection{Algorithms}
The algorithms proposed in this work use the numerical characteristics of the triple $(\mathfrak{g},\mathfrak{h},\mathfrak{l})$ which are a consequence of Proposition \ref{pro1}, Theorem \ref{twkob} and Corollary \ref{wnkob}.
\begin{proposition}\label{prop:numeric-inv}
Assume that $(G,H,L)$ determines a standard compact Clifford-Klein form. Then the following restrictions on the triple $(\mathfrak{g},\mathfrak{h},\mathfrak{l})$ hold:
\begin{enumerate}
  \item By Proposition \ref{pro1} and Corollary \ref{wnkob} we may assume that $\mathfrak{h}$, and $\mathfrak{l}$ are semisimple, and without compact simple ideals,
	\item $\operatorname{rank}_{\mathbb{R}}\mathfrak{g}=\operatorname{rank}_{\mathbb{R}}\mathfrak{h}+\operatorname{rank}_{\mathbb{R}}\mathfrak{l}$, 
	\item $\dim\,\mathfrak{p}=\dim\,\mathfrak{p}_{h}+\dim\,\mathfrak{p}_{l}$,
	\item $\dim\,\mathfrak{k}>\dim\,\mathfrak{k}_{h}, \ \dim\,\mathfrak{k}>\dim\,\mathfrak{k}_{l},$ since $\mathfrak{h},\mathfrak{l}\subset \mathfrak{g}$.
\end{enumerate}
\end{proposition}
\begin{proof}
The condition 2 follows from Theorem \ref{twkob}. We have $\textrm{rank}_{\mathbb{R}}(\mathfrak{g})\geq\textrm{rank}_{\mathbb{R}}(\mathfrak{h})+\textrm{rank}_{\mathbb{R}}(\mathfrak{l})$, so assume that $L$ acts properly and co-compactly on  $G/H$ and 
$$\textrm{rank}_{\mathbb{R}}\mathfrak{g}>\textrm{rank}_{\mathbb{R}}\mathfrak{h}+\textrm{rank}_{\mathbb{R}}\mathfrak{l}.$$
Let $\mathfrak{g}=\mathfrak{k}+\mathfrak{a}+\mathfrak{n},$ $\mathfrak{h}=\mathfrak{k}_{h}+\mathfrak{a}_{h}+\mathfrak{n}_{h},$ $\mathfrak{h}=\mathfrak{k}_{l}+\mathfrak{a}_{l}+\mathfrak{n}_{l}$ be the Iwasawa decompositions of $\mathfrak{g},$ $\mathfrak{h},$ $\mathfrak{l},$ respectively. Since $d(G)=\dim\,\mathfrak{a}+ \dim\,\mathfrak{n}$ we obtain
$$\dim\,\mathfrak{n}-\dim\,\mathfrak{n}_{h}-\dim\,\mathfrak{n}_{l}<0.$$
After conjugating $H$ and $L$ we have $\mathfrak{n}_{h}, \mathfrak{n}_{l}\subset \mathfrak{n}.$ Thus there exists a non-trivial subalgebra $\mathfrak{n}_{0}:=\mathfrak{n}_{h}\cap \mathfrak{n}_{l}.$ Let $N_{0}\subset N$ be a connected (non-compact) subgroup corresponding to $\mathfrak{n}_{0}.$ Since $N_{0}\subset H,L$ is non-compact,  $L$ can not act properly on $G/H.$ A contradiction.
\end{proof}
Thus, the first step of our computer aided analysis is based on the following plan.
\begin{enumerate}
\item  If $(G,H,L)$ is a triple which yields a standard Clifford-Klein form, then the restrictions 1)-4) of Proposition \ref{prop:numeric-inv} hold.
\item Ideally we should begin with describing all possible triples $(\mathfrak{g},\mathfrak{h},\mathfrak{l})$, where $\mathfrak{g}$ is a simple real  Lie algebra, and $\mathfrak{h},\mathfrak{l}$ are semisimple subalgebras. It is too complicated to do it in full generality, however, one can begin with an easier task of finding all triples $(\mathfrak{g}^c,\mathfrak{h}^c,\mathfrak{l}^c)$. This is known since the work of Dynkin, and one can use the database \cite{dg1}.
\item Thus, one proceeds as follows: writes down all possible triples $(\mathfrak{g},\mathfrak{h},\mathfrak{l})$, where each of the real algebras is a real form of the corresponding complex Lie algebra. Note that it may happen, that {\it there is no true embedding of $\mathfrak{h}\hookrightarrow\mathfrak{g}$, or $\mathfrak{l}\hookrightarrow\mathfrak{g}$}. This is because in general, if $\varepsilon:\mathfrak{g}^c\hookrightarrow \tilde{\mathfrak{g}}^c$ is a complex embedding of arbitrary Lie algebras $\mathfrak{g}^{c}$ and $\tilde{\mathfrak{g}}^{c}$, it may happen that $\varepsilon(\mathfrak{g})\not\subset\tilde{\mathfrak{g}}$.
\item For each of the triples $(\mathfrak{g},\mathfrak{h},\mathfrak{l})$ obtained in Part 3) of our plan, one checks conditions 1)-4) of Proposition \ref{prop:numeric-inv}.
\end{enumerate}    
Using this list we create the following algorithms to obtain the list of possible triples $(\mathfrak{g},\mathfrak{h},\mathfrak{l})$ (note again that at this stage we do not check if $\mathfrak{h},\mathfrak{l}$ can be realized as subalgebras of $\mathfrak{g}$).
\vskip10pt
\begin{algorithm}[H]
  \caption{\tt PotentialSubalgebras($\mathfrak{g}$)}
  \footnotesize
  \label{alg1}
  \tcc{
  $\mathfrak{g}$ - non-compact real exceptional simple Lie algebra. 
   Return list $L$ of "potential" semisimple subalgebras of $\mathfrak{g}$.}
  \Begin{
  let $H1C$ be a set of all semisimple subalgebras of $\mathfrak{g}^c$ (use SLA plugin \cite{sla}) \;
  set $H2:= \emptyset$, $H:=\emptyset$\;
  \ForAll{$\mathfrak{h}^c \in H1C$}{
    add all real forms of $\mathfrak{h}^c$ to $H2$\;
    }
    set $\mathfrak{k}$ and $\mathfrak{p}$ from Cartan decomposition $\mathfrak{g}=\mathfrak{k}+\mathfrak{p}$\;
  \ForAll{$\mathfrak{h} \in H2$}{
  set $\mathfrak{k}_{\mathfrak{h}}$ and $\mathfrak{p}_{\mathfrak{h}}$ from Cartan decomposition $\mathfrak{h}=\mathfrak{k}_{\mathfrak{h}}+\mathfrak{p}_{\mathfrak{h}}$\;
  \If{$\operatorname{rank}_{\mathbb{R}}\mathfrak{h}<\operatorname{rank}_{\mathbb{R}}\mathfrak{g}$  {\bf and} $\dim \mathfrak{k}_{\mathfrak{h}} <\dim \mathfrak{k}$  {\bf and} $\dim \mathfrak{p}_{\mathfrak{h}} <\dim \mathfrak{p}$ }{add $\mathfrak{h}$ to $H$\;} 
    }
   \Return{$H$}\;
  
  }
\end{algorithm}

\begin{algorithm}[H]
  \caption{\tt PotentialSubalgebraPairs($\mathfrak{g}$)}
  \footnotesize
  \label{alg2}
  \tcc{
  $\mathfrak{g}$ - non-compact real exceptional simple Lie algebra. 
   Return list $HL$ of "potential" semisimple subalgebra pairs ($\mathfrak{h}$,$\mathfrak{l}$).}
  \Begin{
  set $HL:=\emptyset$\;
  set $H:=\mathtt{PotentialSubalgebras} (\mathfrak{g})$, $L:=\mathtt{PotentialSubalgebras} (\mathfrak{g})$\;
  set $\mathfrak{p}$ from Cartan decomposition $\mathfrak{g}=\mathfrak{k}+\mathfrak{p}$\;
  \ForAll{$\mathfrak{h} \in H$}{
    \ForAll{$\mathfrak{l} \in L$}{
    set  $\mathfrak{p}_{\mathfrak{h}}$ from Cartan decomposition $\mathfrak{h}=\mathfrak{k}_{\mathfrak{h}}+\mathfrak{p}_{\mathfrak{h}}$\;
 set  $\mathfrak{p}_{\mathfrak{l}}$ from Cartan decomposition $\mathfrak{l}=\mathfrak{k}_{\mathfrak{l}}+\mathfrak{p}_{\mathfrak{l}}$\;
 \If{$\operatorname{rank}_{\mathbb{R}}\mathfrak{h}+\operatorname{rank}_{\mathbb{R}}\mathfrak{l}=\operatorname{rank}_{\mathbb{R}}\mathfrak{g}$ {\bf and} $\dim \mathfrak{p}_{\mathfrak{h}}+\dim \mathfrak{p}_{\mathfrak{l}}=\dim\mathfrak{p}$ {\bf and} $\dim \mathfrak{p}_{\mathfrak{h}} \leqslant\dim \mathfrak{p}_{\mathfrak{l}}$}{add pair $(\mathfrak{h},\mathfrak{l})$ to $HL$\;}   
    }
 
    }

   \Return{$HL$}\;
  
  }
\end{algorithm}

\vskip10pt
Using these algorithms we obtain Table \ref{tt1} of possible triples $(\mathfrak{g},\mathfrak{h},\mathfrak{l})$ which might yield standard compact Clifford-Klein forms, because they do not violate the restrictions of Proposition \ref{prop:numeric-inv}.
The algorithm also lists the number of linear equivalence classes of embeddings of $\mathfrak{h}$ into $\mathfrak{g}$ and $\mathfrak{l}$ into $\mathfrak{g} .$ This way we obtain the following proposition
\begin{proposition}
For each subalgebra $\mathfrak{h}^c$ or $\mathfrak{l}^c $ (given as the complexifications of the subalgebras $\mathfrak{h},$ $\mathfrak{l}$ listed in Table \ref{tt1}) there is exactly one, up to linear equivalence, embedding of this subalgebra into $\mathfrak{g}^c $ (given as the complexifications of the corresponding algebra $\mathfrak{g}$ listed in Table \ref{tt1}).
\label{propemb}
\end{proposition}

\begin{remark}
A proof of Proposition \ref{propemb} can be given using Table 25 in \cite{dyn}. The Table contains all the simple subalgebras, up to linear equivalence, of the Lie algebras
of exceptional type.
\end{remark}

\begin{table}[h]
\begin{center}

\begin{tabular}{|c|c|c|c|} \hline \hline &
$\mathfrak{g}$ &  $\mathfrak{h}$ &$\mathfrak{l}$\\\hline \hline
1 & $\mathfrak{e}_{6(6)}$& $\mathfrak{so}(2,7)$ &  $\mathfrak{f}_{4(4)}$ \\\hline
2 & $\mathfrak{e}_{6(6)}$ & $\mathfrak{so}(3,7)$ & $\mathfrak{so}(3,7)$ \\\hline
3 & $\mathfrak{e}_{6(6)}$ & $\mathfrak{su}^{\star}(6)$ & $\mathfrak{f}_{4(4)}$\\\hline
4 & $\mathfrak{e}_{6(2)}$& $\mathfrak{so}^{\star}(10)$ & $\mathfrak{so}^{\star}(10)$\\\hline
5 & $\mathfrak{e}_{6(-14)}$	 & $\mathfrak{f}_{4(-20)}$ &  $\mathfrak{f}_{4(-20)}$ \\\hline
6 & $\mathfrak{e}_{6(-26)}$	& $\mathfrak{su}(1,5)$ &  $\mathfrak{f}_{4(-20)}$\\\hline
7 & $\mathfrak{e}_{7(7)}$	& $\mathfrak{su}(3,5)$ &  $\mathfrak{e}_{6(2)}$ \\\hline								
8 & $\mathfrak{e}_{7(7)}$	& $\mathfrak{so}^{\star}(12)$  & $\mathfrak{e}_{6(2)}$ \\\hline
9 & $\mathfrak{e}_{7(-5)}$& $\mathfrak{e}_{6(-14)}$ &  $\mathfrak{e}_{6(-14)}$\\\hline
10 & $\mathfrak{e}_{8(8)}$ &  $\mathfrak{e}_{7(-5)}$ &  $\mathfrak{e}_{7(-5)}$\\\hline
11 & $\mathfrak{f}_{4(4)}$ & $\mathfrak{so}(2,7)$ &  $\mathfrak{so}(2,7)$ \\\hline							
12 & $\mathfrak{g}_{2(2)}$ &  $\mathfrak{su}(1,2)$ & $\mathfrak{su}(1,2)$	\\\hline	\hline
\end{tabular}
\end{center}

\caption{The list of triples $(\mathfrak{g},\mathfrak{h},\mathfrak{l})$ from Algorithm \ref{alg2}. We do not list triples for which $\mathfrak{h}$ or $\mathfrak{l}$ has non-trivial compact ideals.}
\label{tt1}
\end{table}

\subsection{Implementation of algorithms}
We have implemented Algorithm \ref{alg1} and \ref{alg2} in the computer algebra system GAP \cite{gap}. We have also created a special plugin CKForms \cite{ckforms} which uses the following plugins: SLA \cite{sla}, CoReLG \cite{corelg}.

\section{Analyzing Table 1 and the second step of proof of Theorem \ref{tw1}}
 \subsection{The problem} Now we describe the second ingredient of proof of Theorem \ref{tw1}. We see that one has to deal with two issues: 
 \begin{itemize}
 \item is  a triple $(\mathfrak{g},\mathfrak{h},\mathfrak{l})$ of semisimple Lie algebras represented by true monomorphisms of real Lie algebras $\mathfrak{h}\hookrightarrow\mathfrak{g}$ and $\mathfrak{l}\hookrightarrow\mathfrak{g}$?
 \item If yes, does it yield a compact Clifford-Klein form?
 \end{itemize} 
 The description of semisimple subalgebras in complex simple Lie algebras (up to linear equivalence) was done by Dynkin \cite{dyn}. The problem of embeddings of semisimple real Lie algebra into a simple real Lie algebra is more subtle. One may notice the following. If $\mathfrak{g}$ is a real form of $\mathfrak{g}^c$, then for any $\varphi\in\operatorname{Aut}(\mathfrak{g}^c)$ the subalgebra $\varphi(\mathfrak{g})$ is also a real form. Therefore the problem of describing the embeddings $\mathfrak{g}\hookrightarrow \tilde{\mathfrak{g}}$ can be formulated as follows \cite{dg}.
\begin{problem}\label{prob:emb} Let $\varepsilon:\mathfrak{g}^c\hookrightarrow\tilde{\mathfrak{g}}^c$ be an embedding of complex semisimple Lie algebras $\mathfrak{g}^c$ and $\tilde{\mathfrak{g}}^c$, and $\tilde{\mathfrak{g}}_1,...,\tilde{\mathfrak{g}}_l$ be the real forms of $\tilde{\mathfrak{g}}^{c}$ (considered up to conjugacy by some element in $\operatorname{Aut}(\mathfrak{g}^c)$). Let $\mathfrak{g}$ be a real form of $\mathfrak{g}^c$. Find all $i$ such that $\phi(\varepsilon(\mathfrak{g}))\subset\tilde{\mathfrak{g}}_i$ for some $\phi\in\operatorname{Aut}(\mathfrak{g}^c).$
\end{problem}
Thus, we see that if we knew how to solve Problem \ref{prob:emb}, we could apply this method to each of the pair coming from the triples in Table 1, and get the table of triples which indeed may yield Clifford-Klein forms. It is conceivable that this rout might yield important contributions to the entire topic. However, in our case it is sufficient to directly apply  Proposition \ref{kobhyp} to each of the cases contained in Table 1.
\subsection{Hyperbolic orbits and Satake diagrams} 
There is an effective way of classifying hyperbolic orbits in $\mathfrak{g}^{c}$ and in $\mathfrak{g}$ using Dynkin and Satake diagrams  (\cite{ok}, Section 7). We need a brief description of this method.
\noindent
Recall that we are given a Cartan decomposition $\mathfrak{g}=\mathfrak{k}+\mathfrak{p}$ and a split Cartan subalgebra $\mathfrak{t}$ of $\mathfrak{g}$. Notice that $\mathfrak{t}^c$ is a Cartan subalgebra of $\mathfrak{g}^{c} .$ Let $\Delta$ be a root system of $\mathfrak{g}^c$ with respect to $\mathfrak{t}^c$ and choose a set of positive roots $\Delta^{+} $ and a simple root system $\Pi \subset \Delta^{+}.$ We can take $\mathfrak{t}_{r} ,$ a real form of $\mathfrak{t}^c$ given by
$$\mathfrak{t}_{r}=\{X\in\mathfrak{t}^c\,|\,\forall \alpha\in\Delta,\alpha(X)\in\mathbb{R}\}.$$
 Consider the closed Weyl chamber
$$\mathfrak{t}_{r}^+=\{X\in\mathfrak{t}_{r}\,|\,\forall\alpha\in\Delta^+,\alpha(X)\geq 0\}.$$
For every $X \in \mathfrak{t}_{r}$ we define
$$\psi_{X}:\Pi \rightarrow \mathbb{R}, \ \alpha \rightarrow \alpha (X).$$
The above map is called the {\it weighted Dynkin diagram} of $X\in \mathfrak{t}_{r},$ and the value $\alpha (X)$ is the weight of the node $\alpha.$ Since $\Pi$ is a basis of the dual space $\mathfrak{t}_{r}^{\ast}$, the map
$$\psi : \mathfrak{t}_{r} \rightarrow Map(\Pi , \mathbb{R}), \ X \rightarrow \Psi_{X}$$
is an linear isomorphism. One can show that the map
$$\psi |_{\mathfrak{t}_{r}^{+}} : \mathfrak{t}_{r}^{+} \rightarrow \operatorname{Map}(\Pi , \mathbb{R}_{\geq 0}), \ X \rightarrow \psi_{X}$$
is bijective. 
Recall that using $\mathfrak{a}\subset \mathfrak{t}$ and $\Pi$ we can define \textit{the Satake diagram} $S_{\mathfrak{g}}$ of $\mathfrak{g} .$ The Satake diagram $S_{\mathfrak{g}}$ is the Dynkin diagram $D$ of $\mathfrak{g}^c$ with the following modifications
\begin{itemize}
	\item If for $\alpha \in \Pi$ $\alpha |_{\mathfrak{a}}=0$ then the vertex corresponding to $\alpha$ is black.
	\item If for $\alpha , \beta \in \Pi$ $\alpha |_{\mathfrak{a}} = \beta |_{\mathfrak{a}} \neq 0$ then the vertices corresponding to $\alpha$ and $\beta$ are joined with an arrow.
\end{itemize}
 One can show that the isomorphism class of a Satake diagram does not depend on the choice of $\mathfrak{t}$ and $\Pi .$
\begin{definition}
{\rm Let $\psi_{X} \in \operatorname{Map}(\Pi, \mathbb{R})$ be the weighted Dynkin diagram of $\mathfrak{g}^{c}$ and $S_{\mathfrak{g}}$ be the Satake diagram of $\mathfrak{g}$ (constructed as above). We say that $\psi_{X}$ {\it matches} $S_{\mathfrak{g}}$ if all black nodes in $S_{\mathfrak{g}}$ have weights equal to 0 in $\psi_{X}$ and every two nodes joined by an arrow in $S_{\mathfrak{g}}$ have the same weights in $\psi_{X}$.}
\label{dfff}
\end{definition}
We summarize the results of \cite{ok} in the form of a single theorem (see \cite{ok}, Fact 6.1, Proposition 4.5, Theorem 7.4, Theorem 7.5]).
\begin{theorem}[\cite{ok}]\label{thm:ok} Let $\mathfrak{g}$ be semisimple real Lie algebra. The following holds:
\begin{enumerate}
\item there is a bijective correspondence between complex hyperbolic orbits $O_X\subset\mathfrak{g}^c$ and elements $X\in\mathfrak{t}_{r}^+$;
\item the map $\psi:\mathfrak{t}_{r}\rightarrow \operatorname{Map}(\Pi,\mathbb{R})$ induces a linear isomorphism
$$\psi|_{\mathfrak{a}}:\mathfrak{a}\rightarrow\,\{\psi_X\,\,\text{\rm matches}\,S_{\mathfrak{g}}\}, X\rightarrow \psi_X;$$
\item The weighted Dynkin diagram $\psi_X$ of a complex hyperbolic orbit $O_X$ matches $S_{\mathfrak{g}}$ if and only if $O_X$ meets $\mathfrak{g}$. Moreover, $O_X\cap\mathfrak{g}=\operatorname{Ad}_G(X'), X'\in O_X\cap\mathfrak{g}$, that is, $O_X\cap\mathfrak{g}$ is a single hyperbolic orbit in $\mathfrak{g}$.
\end{enumerate}
\end{theorem}
\subsection{Equivalent subalgebras and real hyperbolic orbits }
\begin{theorem}
Assume that $\mathfrak{h}_{1}, \mathfrak{h}_{2}$ are two isomorphic simple subalgebras of non-compact type of $\mathfrak{g}$ different from $\mathfrak{so}(1,7)$ such that
$$Ad_{g}(\mathfrak{h}^{c}_{1})=\mathfrak{h}^{c}_{2}$$
for some $g\in G^c. $ Then there exists a real  hyperbolic orbit in $\mathfrak{g}$ meeting $\mathfrak{h}_{1}$ and $\mathfrak{h}_{2} .$
\label{hyp}
\end{theorem}
\begin{proof}
A case-by-case inspection of the Satake diagrams of simple real Lie algebras shows that there always exists (if $\mathfrak{h}_{2} \neq \mathfrak{so}(1,7)$) a weighted Dynkin diagram $\psi_{X}$ corresponding to a non-trivial hyperbolic orbit in $\mathfrak{h}_{2}^{c}$ with the following properties:
\begin{enumerate}
	\item $\psi_{X}$ matches the Satake diagram of $\mathfrak{h}_{2}$,
	\item $\psi_{X}$ has equal entries in vertices which are conjugated by some automorphism of the Dynkin diagram of $\mathfrak{h}_{2}^{c} ,$
\end{enumerate}
(see Appendix for examples of weighted Dynkin diagrams satisfying the above conditions). It follows from Theorem \ref{thm:ok} that the complex hyperbolic orbit $O_{X}$ meets $\mathfrak{h}_{2}$ (in a single real hyperbolic orbit) and that $X$ is fixed by any  automorphism $\sigma\in\operatorname{Aut}(\mathfrak{h}_2^c)$ induced by the automorphism of the Dynkin diagram. 
Following \cite{ov} (Chapter 3, Section 3.1), we will denote the groups of such automorphisms by  symbols $\operatorname{Aut}(\Pi),\operatorname{Aut}(\Pi_{\mathfrak{g}}),\operatorname{Aut}(\Pi_{\mathfrak{h}_2})$ respectively.
It follows from the assumption that 
$$\operatorname{Ad}_{g}(\mathfrak{h}_1)\subset\mathfrak{h}_2^c,\,\mathfrak{h}_2\subset\mathfrak{h}_2^c.$$
Since $\operatorname{Ad}_g(\mathfrak{h}_1)$ and $\mathfrak{h}_2$ are isomorphic as real forms of $\mathfrak{h}_2^c$, there exists $f\in\operatorname{Aut}(\mathfrak{h}_{2}^c)$ such that 
$$f(\mathfrak{h}_2)=\operatorname{Ad}_g(\mathfrak{h}_1).$$
It is well known that for any connected Lie group $G^c$ with the Lie algebra $\mathfrak{g}^c$ the following equality holds: $\operatorname{Ad}(G^c)=\operatorname{Int}(\mathfrak{g}^c)$. 
Let $H_2^c$ be any connected closed subgroup of $G^c$ corresponding to $\mathfrak{h}_2^c$. Following \cite{ov} (Theorem 3.1, page 105) we can write
$$\operatorname{Aut}(\mathfrak{h}_2^c)=\operatorname{Int}(\mathfrak{h}_2^c)\rtimes\operatorname{Aut}(\Pi_{\mathfrak{h}_2}).$$
Therefore, $f=f_1\circ f_2$, where $f_1\in\operatorname{Int}(\mathfrak{h}_2^c), f_2\in\operatorname{Aut}(\Pi_{\mathfrak{h}_2})$.
Since $X$ is fixed by the automorphisms $\sigma\in\operatorname{Aut}(\Pi_{\mathfrak{h}_2})$, one obtains
$$f(X)=f_1(X).$$
Thus, $f(X)=f_1(X)=\operatorname{Ad}_{h_2}(X)$ for some $h_{2}\in H_{2}^{c}$ and $\operatorname{Ad}_{h_2}(X)\in \operatorname{Ad}_g(\mathfrak{h}_1).$ Therefore the hyperbolic orbit $Ad_{G^c}(X)$ meets $\mathfrak{h}_1 ,$ $\mathfrak{h}_2$ and $\mathfrak{g}.$
But this means that  $\mathfrak{g}\cap O_{X}$ is a single real hyperbolic orbit meeting $\mathfrak{h}_{1}$ and $\mathfrak{h}_{2} ,$ by Theorem \ref{thm:ok}.
\end{proof}
\begin{corollary}
Assume that there are two   embeddings of $\mathfrak{e}_{6(2)}$ into $\mathfrak{g}$, such that their complexifications are equivalent. Any real hyperbolic orbit meeting one copy of $\mathfrak{e}_{6(2)}$  meets the second one.
\label{lemhyp}
\end{corollary}
\begin{proof}
It is sufficient to notice that in this case $\mathfrak{h}_1=\mathfrak{h}_2=\mathfrak{e}_{6(2)}$ and that {\it any} weighted Dynkin diagram matching the Satake diagram of $\mathfrak{e}_{6(2)}$ has equal entries in vertices which are conjugate by some automorphism of the Dynkin diagram of $\mathfrak{e}_6$. Thus, the proof of Theorem \ref{hyp} applies.
\end{proof}

\section{Proof of Theorem \ref{tw1}}\label{sec:proof}
We eliminate each case in Table \ref{tt1} using the results from the previous sections and the relations between linear equivalence classes and equivalence classes of embeddings of subalgebras into exceptional Lie subalgebras found in \cite{min}. In greater detail, we use the following.

\begin{theorem}[\cite{min}, Theorem 7]
Let $\tilde{\mathfrak{g}}^{c}$ be a complex simple exceptional Lie algebra, and $\mathfrak{g}^{c}$ a semisimple Lie algebra. If the pair $(\mathfrak{g}^{c},\tilde{\mathfrak{g}}^{c})$ is not listed below then the linear equivalence class of an arbitrary embedding $\varepsilon(\mathfrak{g}^{c})\subset\tilde{\mathfrak{g}}^{c}$ consists of a single class of equivalent embeddings.
$$\tilde{\mathfrak{g}}^{c}= \mathfrak{e}_{6}^{c}, \ \ \mathfrak{g}^{c}=\mathfrak{sl}_{3}^{c}, \mathfrak{so}(5,\mathbb{C}), \mathfrak{g}_{2}^{c},$$
$$\tilde{\mathfrak{g}}^{c}=\mathfrak{e}_{7}^{c}, \ \ \mathfrak{g}^{c}=\mathfrak{sl}_{3}^{c}+\mathfrak{sl}_{2}^{c}+\mathfrak{sl}_{2}^{c}, \mathfrak{sl}_{3}^{c}+\mathfrak{sl}_{2}^{c}+\mathfrak{sl}_{2}^{c}+\mathfrak{sl}_{2}^{c},$$
$$\tilde{\mathfrak{g}}^{c}=\mathfrak{e}_{8}^{c}, \ \ \mathfrak{g}^{c}=\mathfrak{h}_{1}^{c}+\mathfrak{sl}_{3}^{c},\mathfrak{so}(5,\mathbb{C})+\mathfrak{sl}_{4}^{c},\mathfrak{sl}_{3}^{c},$$
where $\mathfrak{h}_{1}^{c}=\mathfrak{sl}_{3}^{c},\mathfrak{so}(5,\mathbb{C}),\mathfrak{g}_{2}^{c},\mathfrak{so}(8,\mathbb{C}),\mathfrak{sl}_{2}^{c}+\mathfrak{sl}_{2}^{c}+\mathfrak{sl}_{2}^{c},\mathfrak{sl}_{2}^{c}+\mathfrak{sl}_{2}^{c}+\mathfrak{sl}_{2}^{c}+\mathfrak{sl}_{2}^{c}.$
\label{twmin}
\end{theorem}
Now we perform a case-by-case analysis.

\begin{enumerate}
	\item The cases 1-3, 6 can be eliminated using Theorem \ref{thm:bo} below.
	\begin{theorem}[\cite{bo}, Theorem 2]\label{thm:bo}
	If $L$ acts properly on $G/H$ then
	$$\operatorname{rank}_{\operatorname{a-hyp}}(\mathfrak{l})+\operatorname{rank}_{\operatorname{a-hyp}}(\mathfrak{h})\leq \operatorname{rank}_{\operatorname{a-hyp}}(\mathfrak{g}).$$

	\end{theorem}
	But $\operatorname{rank}_{\textrm{a-hyp}}(\mathfrak{e}_{6(6)})=4,$  $\operatorname{rank}_{\textrm{a-hyp}}(\mathfrak{e}_{6(-26)})=1$ and the a-hyperbolic ranks of the subalgebras in cases 1-3, 6 are equal to their real ranks. One easily checks that in each of these cases:
		$$\operatorname{rank}_{\textrm{a-hyp}}(\mathfrak{l})+\operatorname{rank}_{\textrm{a-hyp}}(\mathfrak{h})> \operatorname{rank}_{\textrm{a-hyp}}(\mathfrak{g}).$$
	\item The cases 4, 5, 9-12 can be eliminated as follows. Since the embeddings of $\mathfrak{h}^c$ and $\mathfrak{l}^c$ are linearly equivalent, they are equivalent (as thay are not listed in Theorem \ref{twmin}). 
	It follows from Theorem \ref{hyp} that there exists a real hyperbolic orbit meeting $\mathfrak{h}$ and $\mathfrak{l}$ (as $\mathfrak{h}$ and $\mathfrak{l}$ are isomorphic). Using Proposition \ref{kobhyp} we see that $L$ cannot act properly on $G/H$.
	
	
	\item The cases 7, 8. As in the previous cases there is only one linear equivalence class of embeddings of $\mathfrak{e}_{6(2)}^c$ in $\mathfrak{e}_{7(7)}^c $ and thus there is only one equivalence class of embeddings of $\mathfrak{e}_{6(2)}^c$ in $\mathfrak{e}_{7(7)}^c $ (again, by Theorem \ref{twmin}). Let $(\mathfrak{e}_{7(7)},\mathfrak{e}_{6(2)}\oplus \mathfrak{so}(2))$ be the symmetric pair. By \cite{tojo} we know that the symmetric space corresponding to this pair does not admit standard compact Clifford-Klein forms. Therefore for every potential triple $(\mathfrak{e}_{7(7)},\mathfrak{e}_{6(2)}\oplus \mathfrak{so}(2), \mathfrak{l})$ (that is a triple that fulfills the requirements 2. and 3. of Proposition \ref{prop:numeric-inv}) there exists a real hyperbolic orbit in $\mathfrak{g}$ that meets $\mathfrak{l}$ and $\mathfrak{e}_{6(2)}$ (we use a fact that a hyperbolic orbit in $\mathfrak{g}=\mathfrak{k}+\mathfrak{p}$ always meets $\mathfrak{p},$ and thus every hyperbolic orbit meeting $\mathfrak{e}_{6(2)}\oplus \mathfrak{so}(2)$ meets $\mathfrak{e}_{6(2)}$). But by Corollary \ref{lemhyp} this orbit meets every copy of $\mathfrak{e}_{6(2)}$. Thus the cases 7, 8 can not yield standard compact Clifford-Klein forms.

\end{enumerate}

\section{Appendix}

In this section we present examples of weighted Dynkin diagrams of hyperbolic orbits in $\mathfrak{h}_{2}^{c}$ satisfying conditions
\begin{enumerate}
	\item $\psi_{X}$ matches the Satake diagram of $\mathfrak{h}_{2}$,
	\item $\psi_{X}$ has equal entries in vertices which are conjugated by some automorphism of the Dynkin diagram of $\mathfrak{h}_{2}^{c} ,$
\end{enumerate}
stated in the proof of Theorem \ref{hyp}. Let $a,b,c,d \in \mathbb{R}$

\begin{table}
\begin{tabular}{|c|c|c|} \hline
Real form & Satake diagam & Weighted Dynkin diagram \\\hline
$\mathfrak{so}(1,7)$ &\dynkin[edge length=.75cm]{D}{o**.*}&--\\\hline
$\mathfrak{so}(2,6)$ &\dynkin[edge length=.75cm]{D}{*o*.o}&\dynkin[labels={0,a,0,0},edge length=.75cm]{D}{ooo.o}\\\hline
$\mathfrak{so}(3,5)$ &\begin{dynkinDiagram}[edge length=.75cm]{D}{ooo.o}\draw[thick, black,latex'-latex']
(root 3) to [out=-90, in=90] (root 4);\end{dynkinDiagram}&\begin{dynkinDiagram}[labels={b,a,b,b},edge length=.75cm]{D}{ooo.o}\end{dynkinDiagram}\\\hline
$\mathfrak{so}(4,4)$ &\dynkin[edge length=.75cm]{D}{ooo.o}&\dynkin[labels={b,a,b,b},edge length=.75cm]{D}{ooo.o}\\\hline
$\mathfrak{e}_{6(6)} $ &\dynkin[edge length=.75cm]{E}{oooooo}
&\dynkin[labels={a,b,a,c,a,a},edge length=.75cm]{E}{oooooo}\\\hline
$\mathfrak{e}_{6(2)} $ &\begin{dynkinDiagram}[edge length=.75cm]{E}{oooooo} \draw[thick, black,latex'-latex']
(root 3) to [out=-45, in=-135] (root 5);\draw[thick, black,latex'-latex']
(root 1) to [out=-45, in=-135] (root 6);\end{dynkinDiagram}&\dynkin[labels={a,b,a,c,a,a},edge length=.75cm]{E}{oooooo}\\\hline
$\mathfrak{e}_{6(-14)} $ &\begin{dynkinDiagram}[edge length=.75cm]{E}{oo***o} \draw[thick, black,latex'-latex']
(root 1) to [out=-45, in=-135] (root 6);\end{dynkinDiagram}&\dynkin[labels={a,b,0,0,0,a},edge length=.75cm]{E}{oooooo}\\\hline
$\mathfrak{e}_{6(-26)} $ &\begin{dynkinDiagram}[edge length=.75cm]{E}{oo***o} \end{dynkinDiagram}&\dynkin[labels={a,b,0,0,0,a},edge length=.75cm]{E}{oooooo}\\\hline
$\mathfrak{f}_{4(4)} $ &\dynkin[edge length=.75cm]{F}{oooo}&\dynkin[labels={a,b,c,d},edge length=.75cm]{F}{oooo}\\\hline
$\mathfrak{f}_{4(-20)} $ &\dynkin[edge length=.75cm]{F}{***o}&
\dynkin[labels={0,0,0,a},edge length=.75cm]{F}{oooo}\\\hline

\end{tabular}
\end{table}

Faculty of Mathematics and Computer Science
\vskip6pt
University of Warmia and Mazury
\vskip6pt
S\l\/oneczna 54, 10-710 Olsztyn, Poland
\vskip6pt
e-mail adresses:
\vskip6pt
mabo@matman.uwm.edu.pl (MB)
\vskip6pt
piojas@matman.uwm.edu.pl (PJ)
\vskip6pt
aleksymath46@gmail.com (AT)

\end{document}